\documentclass[12pt]{article}
\usepackage{caption,cite}

\usepackage{amsthm,amsmath,amssymb}
\usepackage{graphicx}
\usepackage[colorlinks=true,citecolor=black,linkcolor=black,urlcolor=blue]{hyperref}
\usepackage{tikz}
\newtheorem{theorem}{Theorem}

\newtheorem{cor}[theorem]{Corollary}

\theoremstyle{definition}

\newtheorem{conjecture}[theorem]{Conjecture}

\theoremstyle{remark}

\DeclareMathOperator{\cc}{cc}
\DeclareMathOperator{\cp}{cp}
\DeclareMathOperator{\bc}{bc}
\DeclareMathOperator{\bp}{bp}

\DeclareMathOperator{\scc}{scc}	
\DeclareMathOperator{\scp}{scp}
\DeclareMathOperator{\E}{\mathbf{E}}

\title{\bf Sigma clique covering  of graphs}
\author{Akbar Davoodi \qquad Ramin Javadi \qquad Behnaz Omoomi\\[4pt]%\thanks{Supported by NASA grant ABC123.}\\
\small Department of Mathematical Sciences\\[-0.8ex]
\small Isfahan University of Technology\\[-0.8ex] 
\small 84156-83111, Isfahan, Iran\\
}

\date{}

\begin{document}

\maketitle

\begin{abstract}
 The {sigma clique cover number} (resp. {sigma clique partition number}) of graph $G$, denoted by $\scc(G)$ (resp. $\scp(G)$), is defined as the smallest integer $k$ for which 
there exists a collection of cliques of $G$, covering (resp. partitioning) all edges of $G$ such that the sum of sizes of the cliques is at most $k$.
%It is proved that for every graph $G$ on $n$ vertices, $\scp(G) \leq n^2/2$. 
In this paper, among some results we provide some tight bounds for $\scc$ and $\scp$. 

  % keywords are optional
  \bigskip\noindent \textbf{Keywords:} clique covering; clique partition; sigma clique covering; sigma clique partition; set intersection representation; set system.
%\textbf{MSC:} 05C70,05C62, 05C69, 05D05,05B05
\end{abstract}
\section{Introduction}
Throughout the paper, all graphs are simple and undirected. By a \textit{clique} of a graph $G$, we mean a subset of mutually adjacent vertices of $ G $ as well as its corresponding complete subgraph. The \textit{size} of a clique is the number of its vertices. 
Also, a \textit{biclique} of $G$ is a complete bipartite subgraph of $G$.  A \textit{clique covering } (resp. \textit{biclique covering}) of $G$ is defined as a family of cliques (resp. bicliques) of $G$ such that every edge of $G$ lies in at least one of the cliques (resp. bicliques) comprising this family. A clique (resp. biclique) covering in which each edge belongs to exactly one clique (resp. biclique), is called a \textit{clique {\em (resp.} biclique{\em)} partition}. The minimum size of a clique covering, a biclique covering, a clique partition and a biclique partition of $G$ are called \textit{clique cover number}, \textit{biclique cover number}, \textit{clique partition number} and \textit{biclique partition number} of $G$ and are denoted by $\cc(G), \bc(G), \cp(G)$ and $\bp(G)$, respectively. 

The subject of clique covering has been  widely studied in recent decades. First time, Erd\H os et al. in \cite{erdos66} presented a close relationship between the clique covering and the set intersection representation.  Also, they proved that the clique partition number of a graph on $ n $ vertices cannot exceed ${n^2}/{4} $ (known as Erd\H os-Goodman-P{\'o}sa theorem).
The connections of clique covering and other combinatorial objects have been explored (see e.g. \cite{Wallis87,Rees}). For a survey of the classical results on the clique and biclique coverings see \cite{pullman82,monson95}.

Chung et al. in \cite{chung} and independently Tuza in \cite{tuza} considered a weighted version of the biclique covering. In fact, given a graph $G$, they were concerned with minimizing $\sum_{B\in {\cal B}} |V(B)|$ among all biclique coverings $\mathcal{B}$ of $G$. They proved that every graph on $n$ vertices has a biclique covering such that the sum of number of vertices of these bicliques is $O(n^2/\log n)$ \cite{chung,tuza}. Furthermore, a clique counterpart of weighted biclique cover number has been studied. Following a conjecture by Katona and Tarjan, Chung \cite{chung81}, Gyori and Kostochka \cite{gyori80} and Kahn \cite{kahn81}, independently, proved that every graph on $n$ vertices has a clique partition such that the sum of number of vertices in these cliques is at most $n^2/2$. This can be considered as a generalization of Erd\H os-Goodman-P{\'o}sa theorem.

In this paper, we are concerned with a weighted version of the clique cover number. Let $G$ be a  graph.  The \textit{sigma clique cover number} of $G$, denoted by $\scc(G)$, is defined as the minimum integer $k$ for which there exists a clique covering $\mathcal{C}$ of $G$, such that the sum of its clique sizes is at most $k$. 
For a clique covering $\mathcal{C}$ of a graph $G$ and a vertex $u\in V(G)$, let the \textit{valency} of $u$ (with respect to $\mathcal{C}$), denoted by $\mathcal{V}_\mathcal{C}(u)$, be the number of cliques in $\mathcal{C}$ containing $u$. In fact, 
\[\scc(G)=\min_{\mathcal{C}} \sum_{C\in \mathcal{C}} |C|=\min_{\mathcal{C}} \sum_{u\in V(G)} \mathcal{V}_\mathcal{C}(u),\]
where the minimum is taken over all clique coverings of $G$.
Analogously, one can define \textit{sigma clique partition number} of $G$, denoted by $\scp(G)$. As a matter of fact, the above-mentioned result in \cite{chung81,kahn81,gyori80} states that for every graph $G$ on $n$ vertices, $\scp(G)\leq n^2/2$. 

%Let $G_n$ be the graph on the set $V(G)=\{x_1,\ldots, x_n, y_1,\ldots, y_n\}$, where $x_ix_j, y_i y_j\in E(G)$, for every $i\neq j$. Also, $x_i y_j\in E(G)$, if and only if $i\leq j$. The cliques $C_i:=\{x_i, y_i, y_{i+1},\ldots , y_n\}$, $1\leq i\leq n-1$, along with the clique $C_n:=\{y_n, x_1,\ldots, x_n\}$ form a  clique covering $\mathcal{C}$, where $|\mathcal{C}|\leq n$. Thus, $\cc(G_n)\leq n$. But $\sum_{i=0}^n|C_i|=n(n+5)/2$. ?????????

In order to reveal inherent difference between $\cc(G)$ and $\scc(G)$, we introduce a similar parameter  $\scc'(G)$ which is defined as the minimum of the sum of clique sizes in a clique covering $\mathcal{C}$ achieving $\cc(G)$, i.e. 
\[\scc'(G):= \min\left\{\sum_{C\in \mathcal{C}} |C| \ : \  \mathcal{C} \text{ is a clique covering of $G$ and } |\mathcal{C}|=\cc(G) \right\}.\] 

It is evident that $\scc(G)\leq \scc'(G)$.
In Section~\ref{sec:bound}, first  in Theorem~\ref{thm:scc'}, we will see that  for some classes of graphs $G$, the quotient $\scc'(G)/\scc(G)$ can be arbitrary large. Then,
 we give some general bounds on the sigma clique cover number and the sigma clique partition number. In particular, we prove that if $G$ is a graph on $n$ vertices with no isolated vertex and the maximum degree of the complement of $G$ is $ d-1 $, for some integer $d$, then
$ \scc(G)\leq  cnd\lceil\log \left(({n-1})/{(d-1)}\right)\rceil$, where $c$ is a constant. We conjecture that this upper bound is best up to a constant factor for large enough $ n $. In Section~\ref{sec:ctp}, using a well-known result by Bollob{\'a}s, we prove the correctness of this conjecture for $ d=2 $. In other words, we show that for every even integer $n$, if $G$ is the complement of an induced matching on $n$ vertices, then $\scc(G)\sim n\log n$. Finally, in Section~\ref{sec:cr} we give an interpretation of this conjecture as an interesting set system problem.

%In section~\ref{sec:ctp}, we investigate the sigma clique cover number of the Cocktail party graphs.

\section{Some Bounds}\label{sec:bound}

In this section, first we present a class of graphs  for which the family of clique coverings achieving $\cc(G)$ is disjoint from the family of clique coverings achieving $\scc(G)$. 
Then, we provide several inequalities relating the introduced clique covering parameters. Moreover, we present an upper bound for $\scc(G)$ in terms of the number of vertices and the maximum degree of the complement of $G$.

\begin{theorem}\label{thm:scc'}
There exists a sequence of graphs $\{G_n\}$ such that $\scc'(G_n)/\scc(G_n)$ tends to infinity as $n$ tends to infinity.
\end{theorem}
\begin{proof}
Let $n$ be a positive integer and $G_n$ be a graph on $3n+2$ vertices, such that  $V(G_n)=\{x_0,y_0\}\cup X\cup Y\cup Z$,  where $X=\{x_1,\ldots, x_n\}$, $Y=\{y_1,\ldots, y_n\}$ and $Z=\{z_1,\ldots, z_n\}$ and adjacency is as follows. The sets $X\cup \{x_0\}$, $Y\cup \{y_0\}$ and  $Z$ are three cliques and every vertex in $Z$ is adjacent to every vertex in $X\cup Y$.  Moreover, for all $i,j\in\{1,\ldots, n\}$, $x_i$ is adjacent to $y_j$ if and only if $i=j$ (see Figure~\ref{fig:gn}).
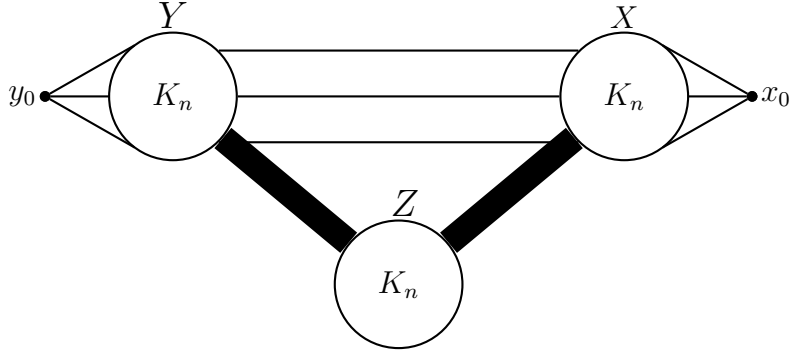
\begin{figure}
\begin{center}
\begin{tikzpicture}
[inner sep=0.4mm, place/.style={circle,draw=black,thick}]
%\node[place,fill=black] (x1) at (3,3) [label=\footnotesize $x_1$] {};
%\node[place] (x2) at (3.5,2.7) [label=\footnotesize $x_2$] {} edge [-,thick] node[auto] {} (x1);
%\node (xdots) at (3.5,2.2) [label=center:\rotatebox{50}{$\cdots$}] {};
%\node[place] (xn) at (3,2) [label=below:\footnotesize $x_n$] {} edge [-,thick] node[auto] {} (x1) edge [-,thick] node[auto] {} (x2);
\node[place] (x0) at (4.7,2.5) [label=right:$x_0$] {} edge [-,thick] node[auto] {} (3.51,3.18) edge [-,thick] node[auto] {} (3.85,2.5) edge [-,thick] node[auto] {} (3.51,1.82);
\node[place] (y0) at (-4.7,2.5) [label=left:$y_0$] {} edge [-,thick] node[auto] {} (-3.51,3.18) edge [-,thick] node[auto] {} (-3.85,2.5) edge [-,thick] node[auto] {} (-3.51,1.82);

%_______________________________________________________________________
\node[place,inner sep=6mm] (X) at (3,2.5) [label= center: $K_n$, label= above:$X$] {};
\node[place,inner sep=6mm] (Y) at (-3,2.5) [label= center: $K_n$, label=above:\large $Y$] {};
\draw [thick] (X) -- (Y);
\draw [thick] (X.north west)  -- (Y.north east);
\draw [thick] (X.south west)  -- (Y.south east);
\node[place,inner sep=6mm] (Z) at (0,0) [label= center: $K_n$, label=above:\large $\ Z$] {}edge [-,line width=10pt] node[auto] {} (X)edge [-,line width=10pt] node[auto] {} (Y);

%_________Y
%\node[place,fill=black] (y1) at (-3,3) [label=\footnotesize $y_1$] {} edge [-,thick] node[auto] {} (x1);
%\node[place] (y2) at (-3.5,2.7) [label=\footnotesize $y_2$] {} edge [-,thick] node[auto] {} (y1)edge [-,thick] node[auto] {} (x2);
%%\node (ydots) at (-3.5,2.2) [label=center:\rotatebox{-50}{$\cdots$}] {};
%\node (dots) at (0,2.4) [label=center:{$\vdots$}] {};
%\node[place] (yn) at (-3,2) [label=below:\footnotesize $y_n$] {} edge [-,thick] node[auto] {} (y1) edge [-,thick] node[auto] {} (y2)edge [-,thick] node[auto] {} (xn);
%\node[place] (y0) at (-4.8,3.3) [label=left:$y_0$] {} edge [-,thick] node[auto] {} (y1) edge [-,thick] node[auto] {} (y2) edge [-,thick] node[auto] {} (yn);
%_______________________Z___________________
%\node[place,fill=black] (z1) at (-.8,.2) [label=below:\footnotesize $z_1$] {};
%\node[place] (z2) at (-.3,-.2) [label=below:\footnotesize $z_2$] {} edge [-,thick] node[auto] {} (z1);
%\node (zdots) at (.3,-.3) [label=center:\rotatebox{28}{$\cdots$}] {};
%\node[place] (zn) at (.7,.2) [label=below:\footnotesize $z_n$] {} edge [-,thick] node[auto] {} (z1) edge [-,thick] node[auto] {} (z2);
%\rotatebox{37}{\draw[fill=black] (4.8,3.4) rectangle (6.8,3.8);} % Z to X
%\rotatebox{-37}{\draw[fill=black] (-1.9,-5.4) rectangle (-3.9,-5);} % Z to Y
%___________________________________________
\foreach \point in {x0,y0}
\fill [black](\point) circle (2pt);
\end{tikzpicture}
\end{center}
\caption{The graph $G_n$.}\label{fig:gn}
\end{figure}

First, note that each clique of $G_n$ covers at most one edge from the set $\{x_iy_i\ :\ 1\leq i\leq n\}\cup \{x_0x_1,y_0y_1\}$. This yields $\cc(G_n)\geq n+2$. 
Now, we show that $G_n$ has a unique clique covering containing  exactly $n+2$ cliques.  Let $\mathcal{C}$ be a clique covering of $G_n$ consisting of  $n+2$ cliques. Assume that the clique $C_i\in \mathcal{C}$ covers the edge $x_i y_i$, for $1\leq i\leq n$, and the cliques $C_{n+1}\in\mathcal{C}$ and $C_{n+2}\in \mathcal{C}$ cover the edges $y_0y_1$ and  $x_0x_1$, respectively. Note that $C_{n+2}\subseteq \{x_0\}\cup X$ and $x_0\not \in \cup_{i=1}^{n+1} C_i$. Therefore, $C_{n+2}=\{x_0\}\cup X$. Similarly, $C_{n+1}=\{y_0\}\cup Y$. Also, we have $x_j, y_j\not\in C_i$, for every $1\leq i\neq j\leq n$. Thus, $C_i=\{x_i,y_i\}\cup Z$, $1\leq i\leq n$. Hence, the clique covering $\mathcal{C}=\{C_i\ :\ 1\leq i\leq n+2\}$ is the unique clique covering of $G_n$ with $n+2$ cliques and then $\cc(G_n)=n+2$. Consequently,
\[\scc'(G_n)=\sum_{C\in\mathcal{C}} |C| = n(n+2)+2(n+1)= n^2+4n+2. \]
On the other hand, the $n+4$ cliques $\{x_0\}\cup X$, $\{y_0\}\cup Y$, $X\cup Z$, $Y\cup Z$ and $\{x_i,y_i\}$, $1\leq i\leq n$, form a clique covering $\mathcal{C}'$ and thus, 
\[\scc(G_n)\leq \sum_{C\in \mathcal{C}'}|C|= 2(n+1)+2(2n)+2n= 8n+2.\]
 Hence, the families of the optimum clique coverings  achieving  $\cc(G_n)$ and $\scc(G_n)$ are disjoint and $\scc'(G_n)/\scc(G_n)$ tends to infinity.  
\end{proof}

In the following, we prove some relations between $\scc(G)$, $\scp(G)$ and $\cp(G)$. 

\begin{theorem}
If $ G $ is a graph with $ m $ edges and  $ \omega(G) $ is the clique number of $G$, then
\begin{enumerate}
\item[{\rm i)}] $\dfrac{2m}{\omega(G)-1}\leq\scc(G)\leq\scp(G)\leq 2m,$
\item[{\rm ii)}] $\dfrac{\scp^2(G)}{2m+\scp(G)}\leq \cp(G).$
\end{enumerate}
Also,  in all relations, the equalities hold for the triangle-free graphs.  
\end{theorem}
\begin{proof}
i) Since the collection of all edges of $G$ is a clique partition for $G$, we have $ \scc(G)\leq\scp(G)\leq2m$. Now, suppose  that $ \cal C $ is a clique covering of $ G $ such that $ \sum_{C\in\cal C}|C|=\scc(G). $ Clearly $ m\leq\sum_{C\in{\cal C}} {\binom{|C|}{2}} $. Hence,
\[2m\leq\sum_{C\in{\cal C}}|C|^2 - \scc(G)\leq(\omega(G)-1)\scc(G) .\]

ii)
Let $\cp(G)=t$ and $\{C_1,\ldots, C_t\}$ be a clique partition of $G$. Then, $m=\sum_{i=1}^{t}{\binom{|C_i|}{2}}$.
Thus,
\begin{align*}
2m&=\sum_{i=1}^{t} |C_i|^2 -\sum_{i=1}^{t} |C_i|\\
&\geq \frac{1}{t}(\sum_{i=1}^{t} |C_i|)^2 - \sum_{i=1}^{t}|C_i|\\%\tag{Cauchy-Schwarz inequality\footnotemark}\\
&\geq\frac{1}{t}\scp^2(G)-\scp(G),
\end{align*}
%\footnotetext{$(\sum_{i=1}^n a_i^2)(\sum_{i=1}^n b_i^2)\geq(\sum_{i=1}^n a_ib_i)^2$}
where the second inequality is due to Cauchy-Schwarz inequality and the last inequality holds because the function $f(x)=\frac{1}{t}x^2-x$ is increasing for $x\geq\frac{t}{2}$ and clearly $\scp(G)\geq \cp(G)=t$.
\end{proof}

For a vertex $u\in V(G)$, let $N_G(u)$ denotes the set of all neighbours of $u$ in $G$ and let $\overline{G}$ stand for the complement of $G$. Moreover, let $\Delta(G)$  be the maximum degree of $G$. Alon in \cite{alon} proved that if $G$ is a graph on $n$ vertices  and $\Delta(\overline{G})=d$, then $\cc(G)=O(d^2\log n)$.	In the following, modifying the idea of Alon, we stablish an upper bound for $\scc(G)$.
\begin{theorem} \label{thm:main}
If $ G $ is a graph on $ n $ vertices with no isolated vertex and $ \Delta(\overline{G})=d-1$, then
\begin{equation} \label{eq:scc}
 \scc(G)\leq (e^2+1)nd\left\lceil\ln\left(\frac{n-1}{d-1}\right)\right\rceil.  
 \end{equation}
\end{theorem}
\begin{proof}
Let $0<p<1$ be a fixed number and  let $ S $ be a random subset of $V(G)$ defined by choosing every vertex $u$ independently with probability $p$. For every vertex $u\in S$, if there exists a non-neighbour of $u$ in $S$, then remove $u$ from $S$. The resulting set is a clique of $G$. Repeat this procedure  $ t $ times, independently, to get $ t $ cliques $ C_1, C_2,\ldots, C_t $ of $ G $. 

Let $ F $ be the set of all the edges which are not covered by the cliques $C_1,\ldots, C_t$. 
For every edge $uv$, using inequality $(1-\alpha)\leq e^{-\alpha}$, we have
\[\Pr(uv\in F)= \left(1- p^2 (1-p)^{|N_{\overline{G}}(u)\cup N_{\overline{G}}(v)|}\right)^t\leq (1-p^2(1-p)^{2(d-1)})^t\leq e^{-tp^2(1-p)^{2(d-1)}}. \]

The cliques $C_1,\ldots, C_t$ along with all edges in $F$ comprise a clique covering of $G$.
Hence,  
\begin{align*}
\scc(G)&\leq \E \left(\sum_{i=1}^t |C_i|+2 |F|\right)\\
&\leq npt+2{\binom{n}{2}}e^{-tp^2(1-p)^{2(d-1)}}. 
%&\leq npt+n(n-1)e^{(\frac{-t}{e^2(d+1)^2})}\nonumber
\end{align*}
Now, set $p:=1/d$. Since $(1-1/d)^{d-1}\geq 1/e$, we have
\[\scc(G)\leq \frac{nt}{d}+ n(n-1) e^{-td^{-2}e^{-2}}. \]
Finally, by setting $ t:=\lceil e^2d^2\ln(\frac{n-1}{d-1})\rceil>0$, we have
\begin{align*}
 \scc(G)&\leq   \frac{n(e^2d^2\ln(\frac{n-1}{d-1})+1)}{d}+n(d-1)  \\
 &\leq nd\left\lceil \ln\left(\frac{n-1}{d-1}\right)\right\rceil
 \left(e^2+\frac{1}{\left\lceil \ln\left(\frac{n-1}{d-1}\right)\right\rceil}\right)\\
& \leq nd\left\lceil \ln\left(\frac{n-1}{d-1}\right)\right\rceil \left(e^2+1\right).
\end{align*}
\end{proof}

The upper bound in \eqref{eq:scc} gives rise to the question that for positive integers $n,d$, how large can be the sigma clique cover number of an $n-$vertex graph where the maximum degree of its complement is $d-1$. A first candidate for graphs with large scc is the family of complete mulipartite graphs. 

For positive integers $n,k$, an \textit{orthogonal array}  OA$(n, k)$ is an $n^2 \times k$ array of elements in $\{1,\dots,n\}$, such that in every two columns each ordered pair $(i,j)$, $1\leq i,j \leq n$, appears exactly once. 
\begin{theorem}\label{prop:turan}
For positive integers $n,d$ with $n\geq 2d$, let $G$ be a complete multipartite graph on $n$ vertices with at least two parts of size $d$ and the other parts of size at most $d$. Then, $\Delta(\overline{G})=d-1$ and   $\scc(G)\geq nd$. Moreover, if $d$ is a prime power and $n\leq d(d+1)$, then $\scc(G)=\scp(G)=nd$.
\end{theorem}
\begin{proof}{
Let $\mathcal{C}$ be a clique covering for $G$. For every vertex $u$, $N_G(u)$ contains a stable set (a set of pairwise nonadjacent vertices) of size $d$. Therefore, $u$ is contained in at least $d$ cliques of $\mathcal{C}$, i.e. the valency of $u$, $\mathcal{V}_{\mathcal{C}}(u)$ is at least $d$. 
Thus, $\scc(G)\geq nd$.

 Now, let $ d $ be a prime power. It is known that there exists an orthogonal array   OA$(d,d+1)$. Let $k=d+1$ and denote the $i$th row of the orthogonal  array  by $a_{i1},a_{i2},\ldots, a_{ik}$. Also, let $H$ be a complete $k-$partite graph on $d(d+1)$ vertices with the parts $V_1,\ldots, V_k$, where $V_j=\{v_{j1},\ldots, v_{jd}\}$, for $1\leq j\leq k$. 
 For each $i\in\{1,\ldots, d^2\}$, the set $C_i:=\{v_{1a_{i1}},v_{2a_{i2}},\ldots, v_{ka_{ik}}\}$ is a clique of $H$. Since in every two columns of OA, each ordered pair $(i,j)$, $1\leq i,j \leq d$, appears exactly once, the collection $\mathcal{C}:=\{C_i\ :\ 1\leq i\leq d^2\}$ forms a clique partition for $H$. Moreover, for every vertex $u\in V(H)$, $\mathcal{V}_{\mathcal{C}}(u)= d$. On the other hand, $G$ is an induced subgraph of $H$. Thus, the collection $\mathcal{C}':=\{C_i\cap V(G)\ :\ 1\leq i\leq d^2\}$ is a clique partition of $G$ and for every vertex $u\in V(G)$, $\mathcal{V}_{\mathcal{C}'}(u)$ is at most $ d$.
 Hence, $\scc(G)\leq\scp(G)\leq nd$.
}\end{proof}

For positive integers $ t,d $, let us denote the complete $ t $-partite graph with each part of size $d$ by $K_t(d)$. Theorem~\ref{thm:main} asserts that $ \scc(K_t(d))\leq cd^2 t\log t $, for some constant $ c $. Although Theorem~\ref{prop:turan} says that $\scc(K_t(d))=d^2t$ when $t\leq (d+1)$ and $d$ is a prime power, we believe that $\scc(K_t(d))$ is much larger when $t$ is sufficiently large. This leads us to the following conjecture.

\begin{conjecture} \label{con:knd}
There exists a function $f$ and a constant $c$, such that for every positive integers $t$ and $d$, if $t\geq f(d)$, then $\scc(K_t(d))\geq cd^2t\log t$.
\end{conjecture}

In fact, if Conjecture~\ref{con:knd} is correct, then the upper bound in \eqref{eq:scc} is best possible up to a constant factor, at least for sufficiently large $n$. In the following section, we will prove that  Conjecture~\ref{con:knd} is true for $d=2$.
\section{Cocktail Party Graphs} \label{sec:ctp}
In this section, we investigate the sigma clique cover number of the Cocktail party graph $ K_t(2) $.
Given a positive integer $t$, the Cocktail party graph $K_t(2)$ is obtained from the complete graph $K_{2t}$ with the vertex set $\{x_1,\ldots ,x_t\}\cup\{y_1,\ldots ,y_t\}$ by removing all the edges $x_iy_i$, $1\leq i\leq t$. 
%If $n$ is an odd positive integer, then $T_n$ is obtained from $T_{n+1}$ by removing a single vertex.

Various clique covering parameters of the Cocktail party graphs have been studied in the literature. 
In $ 1977 $, Orlin \cite{orlin77} asked about asymptotic behaviour of $ \cc(K_t(2)) $, with this motivation that it arises in an optimization problem in Boolean functions theory. He also conjectured that $\cp(K_t(2))\sim t $. Gregory et al. \cite{gregory86} proved that for $ t\geq 4 $, $ \cp(K_t(2))\geq 2t$ and for large enough $ t $, $ \cp(K_t(2))\leq 2t\log\log 2t$. The problem that $ \cp(K_t(2))\sim 2t $ is still an open problem. 
Moreover, Gregory and Pullman \cite{gregory81},  by applying a Sperner-type theorem of Bollob{\'a}s and Sch{\"o}nheim on set systems, proved that for every integer $t$, $ \cc(K_t(2))=\sigma(t) $, where
\begin{equation*}\label{eq:t}
\sigma(t)=\min\left\{k : t \leq \binom{k-1}{\lceil k/2\rceil} \right\}.
\end{equation*}
Furthermore, the authors in \cite{sigma2}, using the pairwise balanced designs, have proved that $\scp(K_t(2))\sim (2t)^{3/2}$.

%Now, for every integer $n$, define
%\begin{equation*}\label{eq:t}
%\delta(n)=\min\left\{k : n \leq \binom{k}{\lfloor k/2\rfloor} \right\},
%\end{equation*}
%and note that $|\sigma(n)-\delta(n)|\leq 1$.

Here, using the following well-known theorem by Bollob{\'a}s, we prove a lower bound for the sigma clique cover number of $ K_t(2) $  which determines the asymptotic behaviour of $\scc(K_t(2))$ and   implies that Conjecture~\ref{con:knd} is true for $ d=2 $. 

%\begin{thm}
{\textbf{Bollob{\'a}s' Theorem.}\cite{bollobas}}
{\em
Let $A_1, \ldots , A_t$ be some sets of size $a_1,\ldots, a_t$, respectively and
$B_1, \ldots , B_t$ be some sets of size $b_1,\ldots, b_t$, respectively, such that $A_i \cap B_j = \emptyset$ if and only if $i = j$. Then
\[\sum_{i=1}^t \binom{a_i+b_i}{a_i}^{-1} \leq 1.\]
}
%\end{thm}

\begin{theorem}\label{thm:ctp}
Let $K_t(2)$ be the Cocktail party graph on $2t$ vertices. Then
\[t\delta(t)\leq \scc(K_t(2))\leq t\sigma(t),\]
where $\sigma(t)$ is defined as above and $\delta(t)=\min\left\{k-1 : t \leq \binom{k}{\lceil k/2\rceil} \right\}$.
\end{theorem}
\begin{proof}
Since $\cc(K_t(2))=\sigma(t)$ and every clique in $K_t(2)$  is of size at most $t$, we have $\scc(K_t(2))\leq t\sigma(t)$.

For the lower bound, assume that $\{C_1,\ldots, C_{k}\}$ is an arbitrary clique covering for $K_t(2)$. For every $ i \in\{1,\ldots,t\}$, define
\[A_i=\{a\ :\ x_i\in C_a\},\quad B_i=\{a\ :\ y_i\in C_a\}.\]
Also, let $a_i=|A_i|, b_i=|B_i|$ and $c_i=a_i+b_i$. Then for every $i\neq j$, there exists a clique containing the edge $x_iy_j$. Hence, $A_i\cap B_j\neq \emptyset$. Moreover, since no clique contains both vertices $x_i$ and $y_i$, we have $A_i\cap B_i=\emptyset$.

Therefore, by Bollob{\'a}s' theorem, we have
\[\sum_{i=1}^t \binom{a_i+b_i}{a_i}^{-1} \leq 1.\]
For every integer $m$, let $f(m)=\binom{m}{\lceil m/2\rceil}^{-1}$ and $f(x)$ be the linear extension of $f(m)$ in $\mathbb{R}^+$.
Since $f$ is  non-increasing and convex, by Jensen inequality, we have
\[
f\left(\left\lceil \frac{1}{t}\sum_{i=1}^t c_i\right\rceil\right) \leq f\left(\frac{1}{t}\sum_{i=1}^t c_i\right) \leq \frac{1}{t} \sum_{i=1}^t \binom{c_i}{\lceil c_i/2\rceil}^{-1}
\leq \frac{1}{t}\sum_{i=1}^t \binom{a_i+b_i}{a_i}^{-1} \leq  \frac{1}{t}.
\]
Thus,
 $\displaystyle\binom{\lceil\frac{1}{t}\sum_{i=1}^t c_i\rceil}{\lceil\frac{1}{2t}\sum_{i=1}^t c_i\rceil}\geq t$. Therefore, 
\[  \delta(t) \leq \left\lceil\frac{1}{t}\sum_{i=1}^t c_i\right\rceil-1 \leq \frac{1}{t}\sum_{i=1}^t c_i =\frac{1}{t}\sum_{a=1}^k |C_a|.\]
Consequently, $t\delta(t)\leq \scc(K_t(2))$.
\end{proof}
Theorem~\ref{thm:ctp} along with the approximation $\binom{2n}{n} \sim {2^{2n}}/{\sqrt{\pi n}}$ yields the following corollary which proves  Conjecture~\ref{con:knd} for $ d=2 $.
\begin{cor}
For every integer $ t $, $\scc(K_t(2))\sim  t\log t$.
\end{cor}

\section{Concluding Remarks}\label{sec:cr}
In previous section, by considering a clique covering as a set system and applying Bollob{\'a}s' theorem, we proved  Conjecture~\ref{con:knd} for $ d=2 $. In this point of view, this conjecture can be restated as an interesting set system problem and thus it can be viewed as a generalization of Bollob{\'a}s' theorem, as follows. 

\begin{conjecture}\label{con:set}
Let $d\geq 2$, $t\geq 1$ and $\mathcal{F}= \{(A^1_i, A^2_i,\ldots, A^d_i) \ :\  1\leq i\leq t \} $ such that $ A^j_i $ is a set of size $ k_{ij} $ and $ A^j_i \cap A^{j'}_{i'}=\emptyset $ if and only if $ i= i' $ and $ j\neq j' $. Then, there exists a function $ f $ and a constant $ c $, such that for every $ t\geq f(d) $, 
\[\sum_{i,j}{k_{ij}} \geq c d^2 t\log t. \]
\end{conjecture}
Note that Conjecture~\ref{con:set} is true for $d=2$, due to Bollob{\'a}s' theorem.
%____________________________________________________

\end{document}